\documentclass[12pt,onecolumn,a4paper]{article}
\usepackage{amsthm}
\newtheorem{definition}{Definition}[section]
\newtheorem{theorem}{Theorem}[]
\newtheorem{corollary}{Corollary}[]
\newtheorem{lemma}{Lemma}[]
\usepackage{setspace}
\usepackage{amssymb, amsmath, amsfonts}
\usepackage{authblk}
\usepackage{amscd}
\usepackage{graphicx}
\usepackage{tikz}
\usetikzlibrary{arrows,automata,calc,shapes, positioning}
\usepackage{tabularx}
\usepackage{array, makecell}
\usepackage{cite}
\renewcommand{\arraystretch}{2}
\usepackage[symbol]{footmisc}
\usepackage{nicematrix,tikz}
\usetikzlibrary{arrows.meta}
\usepackage{color,soul}
\setulcolor{red}
\usepackage{caption}
\usepackage{rotating}
\usepackage{multirow}

\begin{document}

\title{General Dynamics and Generation Mapping for Collatz-type Sequences}
\author{Gaurav Goyal}
\affil{Department of Mechanical Engineering \\ Indian Institute of Technology, Delhi, India \\ email: 	Gaurav.Goyal@mech.iitd.ac.in}

    \maketitle

    \begin{abstract}

Let an odd integer \(\mathcal{X}\) be expressed as $\left\{\sum\limits_{M > m} b_M 2^M\right\} + 2^m - 1,$ where $b_M \in \{0, 1\}$ and $2^m - 1$ is referred to as the Governor.
In Collatz-type functions, a high index Governor is eventually reduced to $2^1 - 1$.
For the $3\mathcal{Z} + 1$ sequence, the Governor occurring in the Trivial cycle is $2^1 - 1$, while for the $5\mathcal{Z} + 1$ sequence, the Trivial Governors are $2^2 - 1$ and $2^1 - 1$.
Therefore, in these specific sequences, the Collatz function reduces the Governor $2^m - 1$ to the Trivial Governor $2^{\mathcal{T}} - 1$.
Once this Trivial Governor is reached, it can evolve to a higher index Governor through interactions with other terms.
This feature allows $\mathcal{X}$ to reappear in a Collatz-type sequence, since
$2^m - 1 = 2^{m - 1} + \cdots + 2^{\mathcal{T} + 1} + 2^{\mathcal{T}} + (2^{\mathcal{T}} - 1).$
Thus, if $\mathcal{X}$ reappears, at least one odd ancestor of $\left\{\sum\limits_{M > m} b_M 2^M\right\} + 2^{m - 1} + \cdots + 2^{\mathcal{T} + 1} + 2^{\mathcal{T}} + (2^{\mathcal{T}} - 1)$ must have the Governor $2^m - 1$.
Ancestor mapping shows that all odd ancestors of $\mathcal{X}$ have the Trivial Governor for the respective Collatz sequence.
This implies that odd integers that repeat in the $3\mathcal{Z} + 1$ sequence have the Governor $2^1 - 1$, while those forming a repeating cycle in the $5\mathcal{Z} + 1$ sequence have either $2^2 - 1$ or $2^1 - 1$ as the Governor.
Successor mapping for the $3\mathcal{Z} + 1$ sequence further indicates that there are no auxiliary cycles, as the Trivial Governor is always transformed into a different index Governor.
Similarly, successor mapping for the $5\mathcal{Z} + 1$ sequence reveals that the smallest odd integers forming an auxiliary cycle are smaller than $2^5$.
Finally, attempts to identify integers that diverge for the $3\mathcal{Z} + 1$ sequence suggest that no such integers exist.

    \end{abstract}



    \section{Introduction}

The Collatz-type sequences \cite{lagarias20033x+, lagarias20043x+, lagarias20063x+, lagarias2010ultimate} define the following set of rules:

    \begin{itemize}
        \item If $\mathcal{Z}$ is odd, multiply it by an odd integer, and add 1.
        \item If $\mathcal{Z}$ is even, it is divided by 2.
    \end{itemize}

The associated conjecture for the Collatz-type sequence $3\mathcal{Z} + 1$ states that every integer ultimately reduces to 1.
To prove this conjecture, it must be shown not only that the sequence eventually cycles through 1, 4, 2, 1, but also that no integer diverges to infinitely larger values \cite{terras1976stopping, tao2022almost}.

While a complete proof may be elusive, this article attempts to understand the workings of Collatz-type sequences.
For this, odd integers are expressed as modified binary expressions $\left\{\sum\limits_{M > m} b_M 2^M\right\} + 2^m - 1,$ where $b_M \in \{0, 1\}$.
Ancestor and successor maps are constructed based on this form.
The ancestor map provides insight into the odd integers that a repeating sequence can have.
The successor map helps in exploring the possibilities of an auxiliary cycle.

The following list of definitions will be used throughout this article:
    
    \begin{definition}[Governor]
        Given an odd integer $\mathcal{X}$ represented in the modified binary form $\left\{\sum\limits_{M > m} b_M 2^M\right\} + 2^m - 1,$ where $b_M \in \{0, 1\}$ and $m$ is the smallest index of 2, the term $2^m - 1$ is called the Governor as it governs the sequence of odd and even steps.
        The index $m$ is called the Governor index.
        \label{def:Governor}
    \end{definition}

    \begin{definition}[Trivial Governor]
        The Governor appearing in the Trivial cycle of a Collatz-type sequence and is denoted by $2^{\mathcal{T}} - 1$ where $\mathcal{T}$ is called the Trivial index.
        For the Collatz-type $3\mathcal{Z} + 1$, the Trivial Governor is $2^1 - 1$ while for the Collatz-type $5\mathcal{Z} + 1$ it is $2^2 - 1$ and $2^1 - 1$.
    \end{definition}

    \begin{definition}[Governor evolution]
        It is the changes that the Governor goes through as the Collatz-type sequences progresses.
    \end{definition}

    \begin{definition}[Trivial cycle]
        A trivial cycle is a sequence of repeating integers in which the evolution of the Governor is uninfluenced by other terms.
    \end{definition}

    \begin{definition}[Auxiliary cycle]
        The sequence of repeating integers obtained when other terms interact with the Governor and alter its evolution.
    \end{definition}


    \section{General dynamics of the Collatz-type sequences}


    \begin{lemma}[General]
    \label{lemma1}
        When $m > \mathcal{T}$, the general dynamics of Collatz-type functions are such that each new odd integer in the sequence has a smaller Governor index, ultimately leading to the Governor $2^{\mathcal{T}} - 1$.
    \end{lemma}

    \begin{proof}

    \begin{table}[t]
        \centering
        \begin{tabular}{c|c}
            \hline
            $3\mathcal{Z} + 1$ & $5\mathcal{Z} + 1$  \\
            \hline
        $\begin{aligned}
        \mathcal{X}         &=  2^m - 1                     \\
        \mathcal O\{1\}     &=  2^{m + 1} + 2^m - 2         \\
        \mathcal E\{1\}     &=  2^m + 2^{m - 1} - 1         \\
        \mathcal O\{2\}     &=  2^{m + 2} + 2^{m - 1} - 2    \\
        \mathcal E\{2\}     &=  2^{m + 1} + 2^{m - 2} - 1   \\
        \mathcal O\{3\}     &=  2^{m + 2} + 2^{m + 1} + 2^{m - 1} + 2^{m - 2} - 2   \\
        \mathcal E\{3\}     &=  2^{m + 1} + 2^{m} + 2^{m- 2} + 2^{m - 3} - 1
        \end{aligned}$
        &
        $\begin{aligned}
        \mathcal{X}             &=  2^m - 1                     \\
        \mathcal O\{1\}         &=  2^{m + 2} + 2^m - 2^2         \\
        \mathcal E^{(1)}\{1\}   &=  2^{m + 1} + 2^{m - 1} - 2    \\
        \mathcal E^{(2)}\{1\}   &=  2^m + 2^{m - 2} - 1           \\      
        \mathcal O\{2\}         &=  2^{m + 2} + 2^{m + 1} + 2^{m - 2} - 2^2    \\
        \mathcal E^{(1)}\{2\}   &=  2^{m + 1} + 2^{m} + 2^{m - 3} - 2   \\
        \mathcal E^{(2)}\{2\}   &=  2^{m} + 2^{m - 1} + 2^{m - 4} - 1
        \end{aligned}$
        \\
        \hline
        \end{tabular}
        \caption{Governor index reduction with the progress of Collatz sequences}
        \label{tab:GeneralBehavior}
    \end{table}

Consider the evolution of the Governor for the Collatz-type $3\mathcal{Z} + 1$ and $5\mathcal{Z} + 1$ sequences shown in the Table \ref{tab:GeneralBehavior}.

Let $\mathcal{O}\{1\}, \mathcal{O}\{2\}, \ldots$ denote the resulting integer at the end of the odd step while $\mathcal{E}\{1\}, \mathcal{E}\{2\}, \ldots$ denote the resulting integer at the end of the even step.

Table \ref{tab:GeneralBehavior} demonstrates that, once a Collatz-type sequence begins, the Governor of each new odd integer is smaller than that of the previous odd integer.
The following corollaries are deduced:

    \begin{corollary}[General]
        The Governor index is singular, that is, no other term has the same index as the Governor.
        \label{corollary:UniqueIndex}
    \end{corollary}

    \end{proof}
    

    \begin{theorem}[General]
    \label{theorem1}
        When $m = \mathcal{T}$, two scenarios are possible:
        \begin{itemize}
            \item \textbf{Uninterrupted Governor evolution}: The Governor enters the trivial cycle.
            \item \textbf{Interrupted Governor evolution}: The Governor may evolve into a higher-index Governor, or an auxiliary cycle may be obtained.
        \end{itemize}
    \end{theorem}

    \begin{proof}
    
Table \ref{tab:GeneralBehavior_TrivialCycle} illustrates the dynamics when the Governor transforms into the Trivial Governor.
In this context, $2^P$ represents the next smallest index term in the expression of the odd integer.

The trivial cycle causes $2^P$ to evolve into terms with smaller indices until it interacts with the Trivial Governor.
If $2^P$ evolves into a term with an index that matches that of the Trivial Governor, the Trivial Governor's index is updated to a higher value.
Therefore, the presence of Trivial Governors makes the Corollary \ref{corollary:UniqueIndex} invalid.

In cases where there are multiple Trivial Governors, such as in the $5\mathcal{Z} + 1$ sequence, $2^P$ may evolve into the smaller Trivial Governor.
This evolution alters the sequence of Trivial Governors within the trivial cycle, leading to an auxiliary cycle.

These observations are summarized in the following corollaries:

    \begin{table}[t]
        \centering
        \begin{tabular}{|c|c|}
            \hline
            $3\mathcal{Z} + 1$ & $5\mathcal{Z} + 1$  \\
            \hline
        $\begin{aligned}
        \mathcal{X}_m                 &=  2^P + 2^1 - 1                     \\
        \mathcal O\{m\}             &=  2^{P + 1} + 2^P + 2^2         \\
        \mathcal E^{(1)}\{m\}       &=  2^P + 2^{P - 1} - 2^1         \\
        \mathcal E^{(2)}\{m\}       &=  2^{P - 1} + 2^{P - 2} + 2^1 - 1    \\
        \end{aligned}$
        &
        $\begin{aligned}
        \textbf{$m$ is odd} \\
        \mathcal{X}_{\frac{m + 1}{2}}             &=  2^P + 2^1 - 1                     \\
        \mathcal O\left\{\frac{m + 1}{2}\right\}         &=  2^{P + 2} + 2^P + 2^3 + 2^1 - 2^2         \\
        \mathcal E\left\{\frac{m + 1}{2}\right\}   &=  2^{P + 1} + 2^{P - 1} + 2^2 - 1   \\
        \hline
        \textbf{$m$ is even}    \\
        \mathcal{X}_{\frac{m}{2}}             &=  2^P + 2^2 - 1           \\      
        \mathcal O\left\{\frac{m}{2}\right\}         &=  2^{P + 2} + 2^P + 2^4    \\
        \mathcal E^{(4)}\left\{\frac{m}{2}\right\}   &=  2^{P - 2} + 2^{P - 4} + 2^1 - 1   \\
        \end{aligned}$
        \\
        \hline
        \end{tabular}
        \caption{Dynamics of the Collatz-type sequences when the Governor term vanishes.}
        \label{tab:GeneralBehavior_TrivialCycle}
    \end{table}

    \begin{corollary}[$3\mathcal{Z} + 1$]
        The integer obtained at the $m^{th}$ $\mathcal{O}$ step is followed by two even steps.
        If the Trivial Governor is uninterrupted, the Collatz sequence starts to shrink as the Trivial Governor enters the Trivial cycle.
    \end{corollary}

    \begin{corollary}[$5\mathcal{Z} + 1$]
        If $m$ is odd then the integer obtained at $\mathcal{E}^{(1)}\{\frac{m + 1}{2}\}$ is odd and is followed by an odd step.
    \end{corollary}

    \begin{corollary}[$5\mathcal{Z} + 1$]
        If $m$ is even then the integer obtained at $\mathcal{E}^{(2)}\{\frac{m}{2}\}$ is even and additional even steps are followed until an odd integer is obtained.
    \end{corollary}

    \begin{corollary}[General]
        The Governor index is shared by another term only if the Governor is the Trivial Governor.
        \label{corollary:GovernorIndexShare}
    \end{corollary}
        
    \end{proof}


   \begin{lemma}[General]
    \label{lemma:steiner}
        In the Collatz-type sequences where $m > \mathcal{T}$, the odd integer $\mathcal{X} = \left\{\sum\limits_{M > m} b_M 2^M\right\} + 2^m - 1,$ repeats if the odd integer $\left\{\sum\limits_{M > m} b_M 2^M\right\} + 2^{m - 1} + \cdots 2^{\mathcal{M}} + 2^{\mathcal{M}}- 1$ appears.
    \end{lemma}

    \begin{proof}
    
Let the integer $\mathcal{X}$ repeats.
The repetition is possible only when the term $2^m - 1$ reappears as the Governor of another odd integer.

Lemma \ref{lemma1} states that when $m > \mathcal{T}$, the Collatz function produces a sequence where the Governor index of each new odd integer is smaller than that of the previous odd integer.
Corollary \ref{corollary:UniqueIndex} further asserts that the Governor index is singular.

It implies that the odd integer $\mathcal{X}$ cannot reappear as $\sum \limits_{M > m} b_M2^M + 2^m - 1$.
Instead, for $\mathcal{X}$ to reappear while satisfying the constraints of Lemma \ref{lemma1} and Corollary \ref{corollary:UniqueIndex}, it must be structured as follows when it reappears:

    \begin{align*}
        \mathcal{X} =   \left\{\sum \limits_{M > m} b_M2^M\right\} + 2^{m - 1} + \cdots + 2^{\mathcal{M}} + 2^{\mathcal{M}}- 1
    \end{align*}

    \end{proof}


    \begin{lemma}[General]
        The structure of the odd integer $\mathcal{X}$ when it reappears, given by $\left\{\sum \limits_{M > m} b_M2^M\right\} + 2^{m - 1} + \cdots 2^{\mathcal{M}} + 2^{\mathcal{M}}- 1$, is possible only when $2^{\mathcal{M}} - 1$ is the Trivial Governor, $2^{\mathcal{T}} - 1$.
    \end{lemma}

    \begin{proof}
    
When $\mathcal{X}$ reappears, its Governor is $2^{\mathcal{M}} - 1$.
This Governor index is also shared by another term.
Corollary \ref{corollary:UniqueIndex} prohibits this situation.
However, Corollary \ref{corollary:GovernorIndexShare} allows sharing of Governor index if and only if the Trivial Governor is present.
Therefore, for the odd integer $\mathcal{X}$ to repeat, the Governor $2^{\mathcal{M}} - 1$ must be the Trivial Governor, $2^{\mathcal{T}} - 1$.

\textbf{Remarks}: In hindsight, the statement of Lemma \ref{lemma:steiner} is incorrect as it contradicts Steiner's findings \cite{steiner1977theorem}.
However, the final result, where $\mathcal{X}$ repeats after the Trivial Governor appears, aligns with Steiner's conclusion.
    
    \end{proof}


    \section{Generation mapping for $3\mathcal{Z} + 1$ sequence}


    \subsection{Ancestor Map}

    \begin{theorem}[$3\mathcal{Z} + 1$]
    \label{Theorem:AncestorMap_3z}
        All the odd integers in the repeating cycle have the Trivial Governor $2^1 - 1$.
    \end{theorem}

    \begin{proof}

    \begin{table}[t]
    \centering
    \resizebox{\textwidth}{!}{
        \begin{NiceTabular}{cccccccccccccc}
                $\mathcal{X}$ & $2^6$ & + & $2^5$ & + & $2^4$ & + & $2^3$ & + & $2^2$ & + & $2^1$ & + & \ul{$2^1 - 1$} \\
                Even ancestor & $2^7$ & + & $2^6$ & + & $2^5$ & + & $2^4$ & + & $2^3$ & + & $2^2$ & + & \ul{$2^1$}  \\
                Even ancestor & $2^8$ & + & $2^7$ & + & $2^6$ & + & $2^5$ & + & $2^4$ & + & $2^3$ & + &  \ul{$2^2$}  \\
                \vdots &&&&&&&&&&&&&\\
                Even ancestor & $2^{i + 6}$ & + & $2^{i + 5}$ & + & $2^ {i + 4}$ & + & $2^{i + 3}$ & + & $2^{i + 2}$ & + & $2^{i + 1}$ & + &  \ul{$2^i$}  \\
                Odd ancestor &  & $2^{i + 5}$ &  & + &  & $2^{i + 3}$ &  & + &  & $2^{i + 1}$ &  & + &  \ul{$2^{\mu} - 1$}  \\
    \CodeAfter
        \begin{tikzpicture} [shorten < = 1mm, shorten > = 1mm, blue, <-]
            \draw (1-2) -- (2-2);
            \draw (1-4) -- (2-4);
            \draw (1-6) -- (2-6);
            \draw (1-8) -- (2-8);
            \draw (1-10) -- (2-10);
            \draw (1-12) -- (2-12);
            \draw (1-14) -- (2-14);
            \draw (2-2) -- (3-2);
            \draw (2-4) -- (3-4);
            \draw (2-6) -- (3-6);
            \draw (2-8) -- (3-8);
            \draw (2-10) -- (3-10);
            \draw (2-12) -- (3-12);
            \draw (2-14) -- (3-14);
            \draw (3-2) -- (5-2);
            \draw (3-4) -- (5-4);
            \draw (3-6) -- (5-6);
            \draw (3-8) -- (5-8);
            \draw (3-10) -- (5-10);
            \draw (3-12) -- (5-12);
            \draw (3-14) -- (5-14);
            \draw (5-2) -- (6-3);
            \draw (5-4) -- (6-3);
            \draw (5-6) -- (6-7);
            \draw (5-8) -- (6-7);
            \draw (5-10) -- (6-11);
            \draw (5-12) -- (6-11);
            \draw (5-14) -- (6-14);
        \end{tikzpicture}
    \end{NiceTabular}}
    \caption{Ancestor map of $\mathcal{X}$ upon reappearance for the Collatz-type 3$\mathcal{Z}$+1 sequence. Only a few indices are shown.}
    \label{tab:3x+1ancestor} 
    \end{table}

To trace the ancestors of the integer $\mathcal{X}$ upon its reappearance, the generation map is constructed. The ancestors are listed in Table \ref{tab:3x+1ancestor}.

Given that there can be a maximum of three terms in the product $(2^{\mu} - 1)(2 + 1) + 1$, the conditions for existence of an odd ancestor are tabulated in Table \ref{tab:OddAncestorTable_3Z}.

    \begin{table}[t]
        \centering
        \resizebox{\textwidth}{!}{
        \begin{tabular}{|c|c|c|}
            \hline
            Three terms & Two Terms & One term  \\
            \hline
            $\begin{aligned}
            (2^{\mu} - 1)(2 + 1) + 1    &=  2^{i +  2} + 2^{i + 1} + 2^i  \\
            2^{\mu} - 2^{i + 1}         &=  \frac{2^i + 2}{3} 
            \end{aligned}$
    & 
            $\begin{aligned}
            (2^{\mu} - 1)(2 + 1) + 1    &= 2^{i + 1} + 2^i  \\
            2^{\mu} - 2^i               &=  \frac{2}{3}
            \end{aligned}$
    &
            $\begin{aligned}
            (2^{\mu} - 1)(2 + 1) + 1    &= 2^i \\
            2^{\mu + 1} + 2^{\mu}       &= 2^i + 2
            \end{aligned}$
            \\
            \hline
            No solution & No solution & $\mu =1$ and $i = 2$    \\
            \hline
        \end{tabular}}
        \caption{Conditions for an odd ancestor to exist. The solution depicted in column 3 is shown with blue arrows in Table \ref{tab:3x+1ancestor}}
        \label{tab:OddAncestorTable_3Z}
    \end{table}

The odd ancestor according to the solution obtained in Table \ref{tab:OddAncestorTable_3Z} and depicted with blue arrows in Table \ref{tab:3x+1ancestor} is:

    \begin{align*}
        \text{Odd ancestor} = \cdots + 2^7 + 2^5 + 2^3 + 2^1 - 1
    \end{align*}

From the structure of this odd ancestor, it is evident that the Governor of the odd ancestor is the Trivial Governor.
Similarly, all ancestors end in the Trivial Governor.
Therefore, unless the original Governor of $\mathcal{X}$ is also the Trivial Governor, it cannot be part of the repeating cycle.
Additionally, the solution $i = 2$ suggests that two even integers are present between the odd integers in a repeating cycle.

    \end{proof}


    \subsection{Successor Map}

    \begin{theorem}[$3\mathcal{Z} + 1$]
        There exists no auxiliary cycle.
    \end{theorem}

    \begin{proof}
    
Assume there is an auxiliary cycle with at least two distinct odd integers $\mathcal{X}$ and $\mathcal{X}_1$ where $\mathcal{X}_1 > \mathcal{X}$.
Additionally, let $\mathcal{X}$ end in $2^P + 2^1 - 1$.
The relation between $\mathcal{X}$ and $\mathcal{X}_1$ is given by:

    \begin{align*}
        \mathcal{X}_1   &=  \mathcal{OE}(\mathcal{X}) \\
                        &=  \mathcal{OE}(2^P + 2^1 - 1) \\
                        &=  2^{P - 1} + 2^1
    \end{align*}

For $\mathcal{X}_1$ to be odd, $P$ must be 1.
Substituting $P = 1$ into the expression of $\mathcal{X}$, the Governor of $\mathcal{X}$ changes to $2^2 - 1$.
Therefore, $\mathcal{X}$ cannot be part of the repeating cycle, according to Theorem \ref{Theorem:AncestorMap_3z}.

Thus, no auxiliary cycle exist.
        
    \end{proof}


    \section{Generation mapping for $5\mathcal{Z} + 1$ sequence}


    \subsection{Ancestor Map}

    \begin{theorem}[$5\mathcal{Z} + 1$]
    \label{Theorem:AncestorMap_5z}
        All the odd integers in the repeating cycle either have the Trivial Governor $2^2 - 1$ or $2^1 - 1$.
    \end{theorem}

    \begin{proof}
    
The odd integer $\mathcal{X}$ upon reappearance can end in either $2^2 + 2^2 - 1$ or $2^1  + 2^1 - 1$ according to Theorem \ref{theorem1}.
The ancestor map for $\mathcal{X}$ ending in $2^1 + 2^1 - 1$ is shown in Table \ref{tab:5x+1ancestor}.

    \begin{table}[t]
    \centering
    \resizebox{\textwidth}{!}{
        \begin{NiceTabular}{cccccccccccccc}
                $\mathcal{X}$ & $2^6$ & + & $2^5$ & + & $2^4$ & + & $2^3$ & + & $2^2$ & + & $2^1$ & + & \ul{$2^1 - 1$} \\
                Even ancestor & $2^7$ & + & $2^6$ & + & $2^5$ & + & $2^4$ & + & $2^3$ & + & $2^2$ & + & \ul{$2^1$}  \\
                Even ancestor & $2^8$ & + & $2^7$ & + & $2^6$ & + & $2^5$ & + & $2^4$ & + & $2^3$ & + &  \ul{$2^2$}  \\
                \vdots &&&&&&&&&&&&&\\
                Even ancestor & $2^{i + 6}$ & + & $2^{i + 5}$ & + & $2^ {i + 4}$ & + & $2^{i + 3}$ & + & $2^{i + 2}$ & + & $2^{i + 1}$ & + &  \ul{$2^i$}  \\
                Odd ancestor & $2^{i + 6}$ & + & $2^{i + 5}$ &&&  + &&& $2^{i + 2}$  & + & $2^{i + 1}$ & + &  \ul{$2^{\mu} - 1$}  \\
    \CodeAfter
        \begin{tikzpicture} [shorten < = 1mm, shorten > = 1mm, blue, <-]
            \draw (1-2) -- (2-2);
            \draw (1-4) -- (2-4);
            \draw (1-6) -- (2-6);
            \draw (1-8) -- (2-8);
            \draw (1-10) -- (2-10);
            \draw (1-12) -- (2-12);
            \draw (1-14) -- (2-14);
            \draw (2-2) -- (3-2);
            \draw (2-4) -- (3-4);
            \draw (2-6) -- (3-6);
            \draw (2-8) -- (3-8);
            \draw (2-10) -- (3-10);
            \draw (2-12) -- (3-12);
            \draw (2-14) -- (3-14);
            \draw (3-2) -- (5-2);
            \draw (3-4) -- (5-4);
            \draw (3-6) -- (5-6);
            \draw (3-8) -- (5-8);
            \draw (3-10) -- (5-10);
            \draw (3-12) -- (5-12);
            \draw (3-14) -- (5-14);
            \draw (5-2) -- (6-2);
            \draw (5-4) -- (6-4);
            \draw (5-6) -- (6-10);
            \draw (5-10) -- (6-10);
            \draw (5-8) -- (6-12);
            \draw (5-12) -- (6-12);
            \draw (5-14) -- (6-14);
        \end{tikzpicture}
    \end{NiceTabular}}
    \caption{Ancestor map of $\mathcal{X}$ upon reappearance for the Collatz-type 5$\mathcal{Z}$+1 sequence. Only a few indices are shown.}
    \label{tab:5x+1ancestor} 
    \end{table}

Given that there can be a maximum of three terms in the product $(2^{\mu} - 1)(2^2 + 1) + 1$, the conditions for existence of an odd ancestor are tabulated in Table \ref{tab:OddAncestorTable_5Z}.

    \begin{table}[t]
        \centering
        \resizebox{\textwidth}{!}{
        \begin{tabular}{|c|c|c|}
            \hline
            Three terms & Two Terms & One term  \\
            \hline
            $\begin{aligned}
            (2^{\mu} - 1)(2^2 + 1) + 1      &=  2^{i +  2} + 2^{i + 1} + 2^i  \\
            2^{\mu + 2} + 2^{\mu} - 2^2     &=  2^{i +  2} + 2^{i + 1} + 2^i
            \end{aligned}$
    & 
            $\begin{aligned}
            (2^{\mu} - 1)(2^2 + 1) + 1    &= 2^{i + 1} + 2^i  \\
            2^{\mu + 2} + 2^{\mu} - 2^2   &= 2^{i + 1} + 2^i    \\
            
            \end{aligned}$
    &
            $\begin{aligned}
            (2^{\mu} - 1)(2^2 + 1) + 1    &= 2^i \\
            2^{\mu + 2} + 2^{\mu} - 2^2   &= 2^i
            \end{aligned}$
            \\
            \hline
            No solution & $\mu = 1$ and $i = 1$ & $\mu = 2$ and $i = 4$    \\
            \hline
        \end{tabular}}
        \caption{Conditions for an odd ancestor to exist.  The solution depicted in column 3 is shown with blue arrows in Table \ref{tab:5x+1ancestor}.}
        \label{tab:OddAncestorTable_5Z}
    \end{table}

One of the odd ancestor is depicted with blue arrows in Table \ref{tab:5x+1ancestor}.
The odd ancestors according to the solution obtained in Table \ref{tab:OddAncestorTable_5Z} are:

    \begin{align*}
        \text{Odd ancestor} &= \cdots + 2^6 + 2^3 + 2^2 + 2^1 - 1    \\
        &\text{Or}\\
        \text{Odd ancestor} &=  \cdots + 2^9 + 2^7 + 2^5 + 2^2 - 1
    \end{align*}

It is seen that all the odd integers either have the Trivial Governor $2^2 -1$ or $2^1 - 1$.
Therefore, unless the original Governor of $\mathcal{X}$ is a Trivial Governor of $5\mathcal{Z} + 1$, it cannot be part of the repeating cycle.

    \end{proof}


    \subsection{Successor Map}


    \begin{lemma}[$5\mathcal{Z} + 1$]
        The smallest odd integer that is part of a repeating cycle has the Trivial Governor \(2^1 - 1\).
    \end{lemma}

    \begin{proof}
    
Consider a repeating cycle with two odd integers, $\mathcal{X}$ and $\mathcal{X}_1$, where $\mathcal{X}_1 > \mathcal{X}$.
Suppose $\mathcal{X}$ ends with the terms $2^P + 2^2 - 1$.
Since $\mathcal{X}$ is smaller than $\mathcal{X}_1$, the odd step is followed by at most two even steps.

    \begin{align*}
        \mathcal{X}                     &= 2^P + 2^2 - 1 \\
        \mathcal{OE}(\mathcal{X})       &= 2^{P - 1} + 2^3 \\
        \mathcal{OEE}(\mathcal{X})      &= 2^{P - 2} + 2^2
    \end{align*}

If $\mathcal{X}_1 = \mathcal{OE}(\mathcal{X})$, then $P = 1$, and the Trivial Governor of $\mathcal{X}$ is $2^1 - 1$.
If $\mathcal{X}_1 = \mathcal{OEE}(\mathcal{X})$, then $P = 2$, and the Governor of $\mathcal{X}$ is $2^3 - 1$.
According to Theorem \ref{Theorem:AncestorMap_5z}, this does not form a repeating cycle.

Thus, if $\mathcal{X}$ is the smallest odd integer in a repeating cycle, it must have the Trivial Governor $2^1 - 1$.

\end{proof}


    \begin{theorem}[$5\mathcal{Z} + 1$]
        The smallest odd integers that form an auxiliary cycle are smaller than $2^5$.
    \end{theorem}

    \begin{proof}

    \begin{sidewaystable}
        \renewcommand\arraystretch{3}
        \resizebox{\textwidth}{!}{%
        \begin{tabular}{|c|c|c|c|}
            \hline
            \multicolumn{4}{|c|}{$\mathcal{X} = 2^R + 2^1 - 1 $}  \\
            \multicolumn{4}{|c|}{$\mathcal{OEO}(\mathcal{X}) = 2^{R - 1} + 2^4$} \\
            \multicolumn{4}{|c|}{$R$ must be 2, 3 or 4 to make sure $\mathcal{X}$ is the smallest odd integer.} \\
            \hline
            $R = 2$ & $R = 3$ & \multicolumn{2}{|c|}{$R = 4$}   \\           
            $\begin{aligned}
                \mathcal{X}                           &= 2^Q + 2^2 + 2^1 - 1   \\
                \mathcal{OE}(\mathcal{X})             &= 2^{Q - 1} + 2^3 + 2^2 + 2^1 - 1 \\
                \mathcal{OEOE}(\mathcal{X})            &= 2^{Q - 2} + 2^5 + 2^1 - 1  \\
                \mathcal{OEOEOE}(\mathcal{X})         &= 2^{Q - 3} + 2^6 + 2^4 + 2^2 - 1 \\
                \mathcal{OEOEOEOE}^{(5)}(\mathcal{X}) &= 2^{Q - 8} + 2^3 + 2^2 + 2^1 - 1
            \end{aligned}$
            &
            $\begin{aligned}
                \mathcal{X}               &= 2^Q + 2^3 + 2^1 - 1   \\
                \mathcal{OE}(\mathcal{X}) &= 2^{Q - 1} + 2^4 + 2^3 - 1
            \end{aligned}$
            &
            \multicolumn{2}{|c|}{
            $\begin{aligned}
                \mathcal{X}                     &= 2^Q + 2^4 + 2^1 - 1 \\
                \mathcal{OE}(\mathcal{X})       &= 2^{Q - 1} + 2^5 + 2^3 + 2^2 - 1 \\
                \mathcal{OEOEEE}(\mathcal{X})   &= 2^{Q - 4} + 2^4 + 2^3 + 2^2 - 1 \\
                \mathcal{OEOEEEO}(\mathcal{X})  &= 2^{Q - 4} + 2^7 + 2^3
            \end{aligned}$}
            \\
            $\mathcal{OEOEOEOE}^{(5)}(\mathcal{X})$ and $\mathcal{OE}(\mathcal{X})$ are identical if $2^Q$ is disregarded & If $Q = 3$ or $Q = 2$, it violates $Q > R$. Otherwise, it violates Theorem \ref{Theorem:AncestorMap_5z} & \multicolumn{2}{|c|}{$\mathcal{X}$ and $\mathcal{OE}^{(3)}\mathcal{OE}^{(3)}(\mathcal{X})$ are identical if $2^Q$ is disregarded. Otherwise, $Q$ must be 5 or 6.}
            \\
            \hline
            && $Q = 5$ & $Q = 6$    \\
            &&
            $\begin{aligned}
                \mathcal{X}                 &= 2^P + 2^5 + 2^4 + 2^1 - 1 \\
                \mathcal{OE}(\mathcal{X})   &= 2^{P - 1} + 2^6 + 2^5 + 2^4 + 2^3 + 2^2 - 1 \\
                \mathcal{OEOE}(\mathcal{X}) &= 2^{P - 2} + 2^9 + 2^5 + 2^4 - 1
            \end{aligned}$
            &
            $\begin{aligned}
                \mathcal{X}                 &= 2^P + 2^6 + 2^4 + 2^1 - 1 \\
                \mathcal{OE}(\mathcal{X})   &= 2^{P - 1} + 2^7 + 2^6 + 2^3 + 2^2 - 1 \\
                \mathcal{OEOEEE}(\mathcal{X}) &= 2^{P - 4} + 2^6 + 2^5 - 1
            \end{aligned}$
            \\
            && If $P = 3$ or $P = 4$, it violates $P > Q$. Otherwise, it violates Theorem \ref{Theorem:AncestorMap_5z} & If $P = 5$ or $P = 6$, it violates $P > Q$. Otherwise, it violates Theorem \ref{Theorem:AncestorMap_5z} \\
            \hline
        \end{tabular}}
        \caption{Successor map for $5\mathcal{Z} + 1$ sequence.}
        \label{tab:SuccessorMap_5z}
    \end{sidewaystable}

Let $\mathcal{X}$ be the smallest odd integer in an auxiliary cycle, expressed as $2^P + 2^Q + 2^R + 2^1 - 1$, where $P > Q > R$.
The successor map, shown in Table \ref{tab:SuccessorMap_5z}, is constructed according to Theorem \ref{Theorem:AncestorMap_5z} and based on the fact that $\mathcal{X}$ is the smallest odd integer.

The successor map in Table \ref{tab:SuccessorMap_5z} can be interpreted as follows: Successor maps for $R = 2$ and $R = 4$ do not fail and may continue indefinitely while following Theorem \ref{Theorem:AncestorMap_5z} and ensuring that $\mathcal{X}$ remains the smallest odd integer.
However, whether $R = 2$ and $R = 4$ actually form an auxiliary sequence is uncertain without actual calculation or prior knowledge, as is the case here.
In contrast, the successor maps for $R = 3$, $Q = 5$, and $Q = 6$ definitely fail.

Thus, if $\mathcal{X}$ forms an auxiliary cycle, it must be smaller than $2^5$.

    \end{proof}


    \section{Integers that `may' diverge}


    \begin{theorem}[$3\mathcal{Z} + 1$]
        The integers may take a long time to converge to 1, but they cannot escape this convergence indefinitely.
    \end{theorem}

    \begin{proof}

Consider the dynamics of the Collatz-type sequences shown in Table \ref{tab:GeneralBehavior_TrivialCycle}.
For the $5\mathcal{Z} + 1$ sequence, integers may avoid shrinking in value if $m$ is odd, or if the trivial cycle is interrupted, or if the Trivial Governor $2^2 - 1$ is avoided.
In contrast, the $3\mathcal{Z} + 1$ sequence always experiences a shrinking phase, regardless of whether $m$ is even or odd..
 Additionally, there is only one Trivial Governor, and, as discussed in Theorem \ref{theorem1}, all Governors eventually lead to the Trivial Governor.
 This leaves the following options to avoid convergence to 1:

    \begin{itemize}
        \item An infinite value of $m$.
        \item Alternatively, each time the Trivial Governor appears, it is interrupted and transformed into a higher index Governor.
    \end{itemize}

Consider the second approach for $\mathcal{X} = 2^a + 2^3 + 2^2 - 1$.
The Trivial Governor is obtained when

    \begin{align*}
        \mathcal{OEOEE}(\mathcal{X})    &=  2^a + 2^{a - 3} + 2^3 + 2^2 + 2^1 - 1\\
    \end{align*}

If $a = 4$, the Trivial Governor can be transformed into a higher index Governor:

    \begin{align*}
        \mathcal{OEOEE}(\mathcal{X})    &=  2^4 + 2^{4 - 3} + 2^3 + 2^2 + 2 - 1\\
        \mathcal{OEOEE}(\mathcal{X})    &=  2^5 - 1\\
    \end{align*}    

The Governor obtained is higher than the original Governor in $\mathcal{X}$.
However, this method requires an infinite number of carefully chosen terms to indefinitely avoid the Trivial Governor.

Thus, while an infinitely large integer may prolong the convergence to 1, it cannot avoid it entirely.

    \end{proof}


   \section{Conclusion}
   
It is shown that for an odd integer $\mathcal{X}$ to reappear in any Collatz-type sequence, its Governor must be the Trivial Governor of that sequence.
Applied to the $3\mathcal{Z} + 1$ sequence, it is observed that the repeating cycle consists of odd integers that end in $2^1 -1$ and are separated by two even integers, forming the trivial cycle.
No auxiliary cycles are found in this case.

In the $5\mathcal{Z} + 1$ sequence, the odd integers in any repeating cycle either have the Trivial Governor $2^1 - 1$ or $2^2 - 1$.
Furthermore, the smallest odd integers forming an auxiliary cycle are found to be smaller than $2^5$.

Finally, attempts to construct an integer that may diverge in the $3\mathcal{Z} + 1$ sequence reveal that such an integer might not be possible.


   \section*{Data availability statement}

Data availability is not applicable to this article as no new data were created or analysed in this study.

    
    \bibliographystyle{ieeetr}
    \bibliography{B}
    
    \end{document}